\documentclass[oneside,a4paper]{article}

\usepackage{amsmath,amssymb,amsfonts,amsthm,graphicx,authblk}
\usepackage[margin=1in]{geometry}

\usepackage{hyperref}
\hypersetup{
    pdfpagemode=UseNone,colorlinks=true,citecolor=blue,linkcolor=blue,urlcolor=blue,
    pdfstartview=FitH,pdftitle=The sequence of prime gaps is graphic,
    pdfauthor=Péter L. Erdős and Gergely Harcos and Shubha R. Kharel and Péter Maga and Tamás Róbert Mezei and Zoltán Toroczkai}
\usepackage{cleveref}

\usepackage{tikz}
\usetikzlibrary{calc,arrows,snakes,shapes}
\tikzset{every picture/.style=very thin}

\newcommand\DOI[1]{{\tt DOI:\href{https://doi.org/#1}{#1}}}
\newcommand\arXiv[1]{{\tt arXiv:\href{https://arxiv.org/abs/#1}{#1}}}
\newcommand\ISBN[1]{{\tt ISBN:\href{https://isbnsearch.org/isbn/#1}{#1}}}
\newcommand\link[1]{{\tt \href{#1}{#1}}}

\renewcommand{\geq}{\geqslant}
\renewcommand{\leq}{\leqslant}

\newcommand{\D}{\mathbf{D}}
\newcommand{\PD}{\mathbf{PD}}
\newcommand{\RR}{\mathbb{R}}
\newcommand{\ZZ}{\mathbb{Z}}
\newcommand\eps{\varepsilon}
\newcommand\ov{\overline}

\theoremstyle{plain}
\newtheorem{theorem}{Theorem}[section]
\newtheorem{lemma}{Lemma}[section]
\newtheorem{conjecture}{Conjecture}[section]

\theoremstyle{definition}
\newtheorem*{definition}{Definition}

\theoremstyle{remark}
\newtheorem*{remark}{Remark}

\title{\bfseries The sequence of prime gaps is graphic\footnote{
This work was supported in part by NKFIH (National Research, Development and Innovation Office) Grants
SNN~135643 (PLE \& TRM), K~132696 (PLE \& TRM), K~119528 (GH \& PM), KKP~133819 (PM), FK~135218 (PM),
NSF (National Science Foundation) Grant IIS-1724297 (ZT),
and the MTA R\'enyi Int\'ezet Lend\"ulet Automorphic Research Group (GH \& PM).}}

\author[1]{P\'eter~L.~Erd\H{o}s}
\author[1]{Gergely~Harcos}
\author[2]{Shubha~R.~Kharel}
\author[1]{P\'eter~Maga}
\author[1]{Tam\'as~R\'obert~Mezei}
\author[2]{Zolt\'an~Toroczkai}

\affil[1]{\small Alfr\'ed R\'enyi Institute of Mathematics\ (LERN),\protect\\ Re\'altanoda utca 13--15, H-1053 Budapest, Hungary\protect\\\texttt{<erdos.peter,harcos.gergely,maga.peter,mezei.tamas.robert>@renyi.hu}}
\affil[2]{\small Department of Physics, 225 Nieuwland Science Hall, Notre Dame,
    IN 46556, USA\protect\\\texttt{<skharel,toro>@nd.edu}}

\begin{document}

\maketitle

\centerline{\emph{Dedicated to J\'anos Pintz on the occasion of his 71st and 73rd birthdays}}

\begin{abstract}
Let us call a simple graph on $n\geq 2$ vertices a prime gap graph if its vertex degrees are $1$ and the first $n-1$ prime
gaps. We show that such a graph exists for every large $n$, and in fact for every $n\geq 2$ if we assume the Riemann hypothesis. Moreover, an infinite sequence of prime gap graphs can be generated by the so-called degree preserving growth process. This is the first time a naturally occurring infinite sequence of positive integers is identified as graphic. That is, we show the existence of an interesting, and so far unique, infinite combinatorial object.

\smallskip

\end{abstract}

\section{Introduction}\label{sec:intro}

\subsection{The problem}

This paper grew out from an empirical observation by one of us (Z.T.): there are large graphs whose vertex degrees are consecutive members of the sequence of prime gaps. Moreover, such graphs can be generated recursively by the so-called degree preserving growth process~\cite{DPG21}. To turn the observation into precise mathematical statements, we introduce the following definition.

\begin{definition} Let $p_n$ denote the $n$-th prime number, and let $p_0=1$. We call a simple graph on $n\geq 2$ vertices a \emph{prime gap graph} if its vertex degrees are $p_1-p_0,\dotsc,p_n-p_{n-1}$.
\end{definition}

\begin{conjecture}[Toroczkai, 2016]\label{conj1}
For every $n\geq 2$, there exists a prime gap graph on $n$ vertices.
\end{conjecture}

\begin{conjecture}[Toroczkai, 2016]\label{conj2}
In every prime gap graph on $n$ vertices, there exist $(p_{n+1}-p_n)/2$ independent edges.
\end{conjecture}

In fact, as will become clear in the next subsection, \Cref{conj2} implies \Cref{conj1}. By combining techniques from analytic number theory and matching theory, we are able to almost fully settle these conjectures.

\begin{theorem}\label{maintheorem}
Conjectures~\ref{conj1} and~\ref{conj2} are true for every sufficiently large $n$. Assuming the Riemann hypothesis, they are true for every $n\geq 2$.
\end{theorem}

The main input from analytic number theory is an upper bound on the sum of large prime gaps. The main input from matching theory is Vizing's theorem on edge colorings. See Subsection~\ref{newresults} for more details. In the next subsection, we discuss the background and motivation for Theorem~\ref{maintheorem}.

\subsection{Broader context}

Networks are powerful, graph-based representations used in the study of complex systems. They appear in systems ranging from elementary particle interactions, through nucleosynthesis, chemistry, biology (gene interactions, protein interactions, metabolism, physiology),  social sciences (human interactions), infrastructures (transportation, power grid, etc.), ecology (food webs) and climate, to the organization of visible and dark matter in the universe. In this paper we report on a novel family of networks, however, in number theory.

\medskip

In this paper all graphs are \emph{simple}: there are no parallel edges and
loops. The most common characteristic of a graph is its \emph{degree
sequence}: we equip each of the $n$ vertices with a unique label from
$\{1,\dotsc,n\}$, and an integer vector $\D=(d_1,\dotsc,d_n)$ lists the
\emph{degrees} of the corresponding labeled vertices, that is, the number of edges
incident on a given vertex. In chemistry, and in old-fashioned graph theory,
this is called \emph{valency}.

\medskip

The inverse problem is the following: we are given a sequence $\D$ of nonnegative integers,
and we want to know whether there exists a graph with this ensemble as a
degree sequence. When the answer is affirmative, then we call the sequence $\D$
\emph{graphic}. Clearly, if the sequence is graphic, then the sum of its members
must be even. However, it is not self-evident whether a given sequence is graphic. The most well-known characterization of graphic degree sequences is the following theorem:
\begin{theorem}[Erd\H{o}s--Gallai~\cite{EG60}]\label{tm:EG}
Let $d_1\geq\dotsb\geq d_n\geq 0$ be integers. Then the sequence
$(d_1,\dotsc,d_n)$ is graphic if and only if $d_1+\dotsb+d_n$ is
even and for every $k\in\{1,\dotsc,n\}$ we have
\begin{equation}
\sum_{\ell=1}^k d_\ell \leq k(k-1) + \sum_{\ell=k+1}^n \min(k,d_\ell)\;. \label{EG}
\end{equation}
\end{theorem}
It is well-understood that there are exponentially
many different \emph{realizations} for almost every graphic degree sequence. At the
same time, the number of all graphic degree sequences is infinitesimal compared
to the number of integer partitions of the sum of the degrees. More precisely, let $2m$ denote the sum of the degrees (as usual), so that the degree sequence is a \emph{partition} of $2m$. By a difficult result of Pittel~\cite{pittel}, as $m$ tends to infinity, the probability of a random partition of $2m$ being graphic is zero in the limit.

\medskip

\Cref{tm:EG} gives us a means to decide whether the degree sequence is graphic.
It is, however, an entirely different problem to actually construct a realization of a graphic degree
sequence. The simplest way to do that is via the Havel--Hakimi algorithm,
which, in turn, is based on the following observation:
\begin{theorem}[Havel~\cite{Hav} and Hakimi~\cite{Hak}]\label{tm:HH}
Let $d_1 > 0$ and $d_2 \geq \dotsb \geq d_n \geq 0$ be integers. Then $(d_1,\dotsc,d_n)$ is graphic if and only if
$(d_2-1,\dotsc,d_{d_1+1}-1,d_{d_1+2},\dotsc,d_n)$ is graphic.
\end{theorem}
Assume we are given a long sequence of natural numbers $\D$ such that every \emph{initial segment} $\D^n$
formed by the first $n$ elements of $\D$ is graphic
(we restrict this notion to $n\geq 2$). This means, in particular, that $d_n<n$ for $n\geq 2$.
We would like to construct a realization $G_n$ for any $n\geq 2$. Clearly, it can be done by separate applications of the
Havel--Hakimi algorithm for every single $n\geq 2$. But this is rather uneconomical: in
principle, for each new segment we have to restart the algorithm from scratch. Instead, we want to
find a \emph{graph growth dynamics (GGD)} such that $G_{n+1}$ can be obtained from $G_n$ rapidly.

\medskip

There are only a few GGDs in the network science literature as most graph construction models are based on static algorithms.
Arguably, the most popular GGDs is the preferential
attachment algorithm of scale-free networks. However, in this and other GGDs, typically, some of the degrees of the vertices in $G_{n+1}$ are bigger than in the degree sequence of $G_n$, and thus, they are unsuitable for our purposes.

\medskip

Recently, a new network growth dynamics has been introduced: the
\emph{degree-preserving network growth (DPG)} model family
(see~\cite{Degree-math} or~\cite{DPG21}). The DPG-process can be
described as follows: let $G$ be a simple graph with degree sequence $\D$. In what follows, by a \emph{matching} we mean a set of independent edges in the graph, that is, a set of pairwise non-adjacent edges. In a
general step, a new vertex $u$ joins $G$ by removing a $\nu$-element
matching of $G$ followed by connecting $u$ to the vertices incident to the $\nu$ removed edges. The degree of the newly inserted vertex is $d=2\nu$. This step does not join two vertices of $G$ that are non-adjacent, furthermore, the
degrees of vertices in $G$ are not changed. The degree sequence of the newly
generated graph is $\D \circ d$, that is, $d$ is concatenated to the
end of $\D$. This graph operation is called a \emph{degree-preserving step (DP-step)}, and the \emph{DPG-process} repeats such steps iteratively.

\medskip

Returning to our long sequence $\D$ of natural numbers: if, for each step $n$, we
can find a matching of size $d_{n+1}/2$ in $G_n$, then the application of a
DP-step provides a realization $G_{n+1}$ of $\D^{n+1}$. In this case, we say that the pair $(G_n,d_{n+1})$ is \emph{DPG-graphic}. One can ask, why should a suitable matching be found in $G_n$? Actually, this is
not inconceivable, as the following theorem shows:
\begin{theorem}[Theorem~2.5 in \cite{Matching}]\label{tm:weak}
Given a graphic sequence $\D$ of length $n$ and an even integer $2\leq d \leq n$, the
sequence $\D \circ d$ is graphic if and only if $\D$ has a realization with a matching of size $d/2$.
\end{theorem}	
Since every initial segment $\D^n$ is graphic, therefore, for each one there
exists a ``special'' realization $G'_n$ with the requested matching. However, it is not automatic that $G'_n = G_n$.
A natural way to deal with this problem is to add the condition that every realization of $\D^n$ has a matching of size $d_{n+1}/2$.

\medskip

We stress that it is not easy to find an infinite, naturally occurring sequence $\D$ whose initial segments are all graphic, or at least graphic beyond a certain point. As a matter of fact, until now we have only known one such example: when all
elements in the sequence are equal. Such a GGD describes an ever growing regular graph sequence.

\medskip

In this paper, we describe for the first time a nontrivial, naturally arising infinite sequence whose initial segments are all graphic. Furthermore, we show that any realization of the initial segments is admissible for the DPG-algorithm. This sequence is the sequence of prime gaps with a prefix $1$:
\[\PD:=(p_1-p_0,p_2-p_1,\dotsc)=(1,1,2,2,4,2,\dotsc)\]
The prefix $1$ was included to guarantee that the sum of the initial segments is
even. The figure below is an illustration of the DPG-process on prime gap graphs. The independent edges used by the DP-steps are red zigzags.

\bigskip

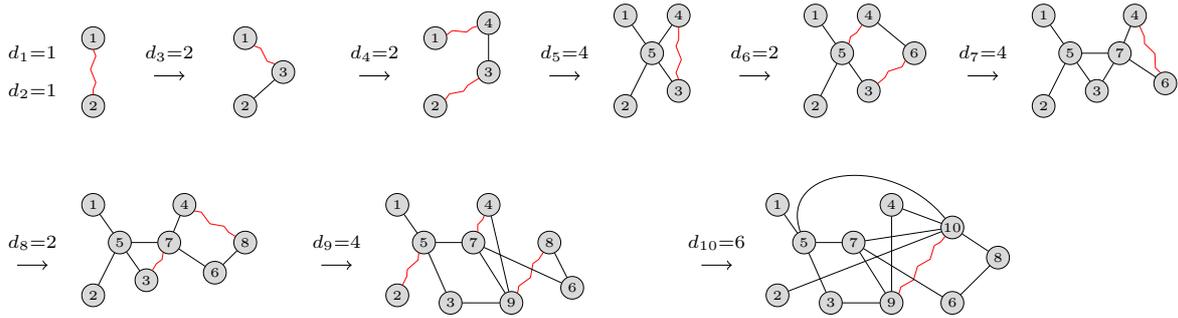
\begin{tikzpicture}
\node at (.2,0) {$\scriptstyle d_1=1$};
\node at (.2,-.5) {$\scriptstyle d_2=1$};
\node at (1,.2) [shape=circle,draw, inner sep=0pt,minimum size=3mm,fill=black!15] {\tiny 1};
\node at (1,-.7) [shape=circle,draw, inner sep=0pt,minimum size=3mm,fill=black!15] {\tiny 2};
\draw [red,snake=zigzag,segment amplitude=1pt,very thin] (1,.05) -- (1,-.55);
\begin{scope}[shift={(2,0)}]
\node at (0,0) {$\scriptstyle d_3=2$};
\draw [->,thin] (-.2,-.3) -- (.2,-.3);
\node at (1,.2) [shape=circle,draw, inner sep=0pt,minimum size=3mm,fill=black!15] (a) {\tiny 1};
\node at (1,-.7) [shape=circle,draw, inner sep=0pt,minimum size=3mm,fill=black!15] (b) {\tiny 2};
\node at (1.5,-.25) [shape=circle,draw, inner sep=0pt,minimum size=3mm,fill=black!15] (c) {\tiny 3};
\draw [red,snake=zigzag,segment amplitude=1pt,very thin] (a) -- (c);
\draw [very thin] (b) -- (c);
\end{scope}

\begin{scope}[shift={(4.5,0)}]
\node at (.2,0) {$\scriptstyle d_4=2$};
\draw [->,thin] (0,-.3) -- (.4,-.3);
\node at (1,.2) [shape=circle,draw, inner sep=0pt,minimum size=3mm,fill=black!15] (a) {\tiny 1};
\node at (1,-.7) [shape=circle,draw, inner sep=0pt,minimum size=3mm,fill=black!15] (b) {\tiny 2};
\node at (1.7,-.25) [shape=circle,draw, inner sep=0pt,minimum size=3mm,fill=black!15] (c) {\tiny 3};
\node at (1.7,.4) [shape=circle,draw, inner sep=0pt,minimum size=3mm,fill=black!15] (d) {\tiny 4};
\draw [red,snake=zigzag,segment amplitude=1pt,very thin] (a) -- (d);
\draw [red,snake=zigzag,segment amplitude=1pt,very thin] (b) -- (c);
\draw [very thin] (c) -- (d);
\end{scope}

\begin{scope}[shift={(7,0)}]
\node at (.2,0) {$\scriptstyle d_5=4$};
\draw [->,thin] (0,-.3) -- (.4,-.3);
\node at (1,.5) [shape=circle,draw, inner sep=0pt,minimum size=3mm,fill=black!15] (a) {\tiny 1};
\node at (1,-.7) [shape=circle,draw, inner sep=0pt,minimum size=3mm,fill=black!15] (b) {\tiny 2};
\node at (1.7,-.5) [shape=circle,draw, inner sep=0pt,minimum size=3mm,fill=black!15] (c) {\tiny 3};
\node at (1.7,.5) [shape=circle,draw, inner sep=0pt,minimum size=3mm,fill=black!15] (d) {\tiny 4};
\node at (1.35,0) [shape=circle,draw, inner sep=0pt,minimum size=3mm,fill=black!15] (e) {\tiny 5};
\draw [very thin] (a) -- (e);
\draw [very thin] (e) -- (d);
\draw [very thin] (e) -- (b);
\draw [very thin] (e) -- (c);
\draw [red,snake=zigzag,segment amplitude=1pt, very thin] (c) -- (d);
\end{scope}

\begin{scope}[shift={(9.5,0)}]
\node at (.2,0) {$\scriptstyle d_6=2$};
\draw [->,thin] (0,-.3) -- (.4,-.3);
\node at (1,.5) [shape=circle,draw, inner sep=0pt,minimum size=3mm,fill=black!15] (a) {\tiny 1};
\node at (1,-.7) [shape=circle,draw, inner sep=0pt,minimum size=3mm,fill=black!15] (b) {\tiny 2};
\node at (1.7,-.5) [shape=circle,draw, inner sep=0pt,minimum size=3mm,fill=black!15] (c) {\tiny 3};
\node at (1.7,.5) [shape=circle,draw, inner sep=0pt,minimum size=3mm,fill=black!15] (d) {\tiny 4};
\node at (1.35,0) [shape=circle,draw, inner sep=0pt,minimum size=3mm,fill=black!15] (e) {\tiny 5};
\node at (2.3,0) [shape=circle,draw, inner sep=0pt,minimum size=3mm,fill=black!15] (f) {\tiny 6};
\draw [very thin] (a) -- (e);
\draw [red,snake=zigzag,segment amplitude=1pt, very thin] (e) -- (d);
\draw [very thin] (e) -- (b);
\draw [very thin] (e) -- (c);
\draw [very thin] (f) -- (d);
\draw [red,snake=zigzag,segment amplitude=1pt, very thin] (f) -- (c);
\end{scope}

\begin{scope}[shift={(12.5,0)}]
\node at (.2,0) {$\scriptstyle d_7=4$};
\draw [->,thin] (0,-.3) -- (.4,-.3);
\node at (1,.5) [shape=circle,draw, inner sep=0pt,minimum size=3mm,fill=black!15] (a) {\tiny 1};
\node at (1,-.7) [shape=circle,draw, inner sep=0pt,minimum size=3mm,fill=black!15] (b) {\tiny 2};
\node at (1.7,-.5) [shape=circle,draw, inner sep=0pt,minimum size=3mm,fill=black!15] (c) {\tiny 3};
\node at (2.2,.5) [shape=circle,draw, inner sep=0pt,minimum size=3mm,fill=black!15] (d) {\tiny 4};
\node at (1.35,0) [shape=circle,draw, inner sep=0pt,minimum size=3mm,fill=black!15] (e) {\tiny 5};
\node at (2.6,.-.4) [shape=circle,draw, inner sep=0pt,minimum size=3mm,fill=black!15] (f) {\tiny 6};
\node at (2,0) [shape=circle,draw, inner sep=0pt,minimum size=3mm,fill=black!15] (g) {\tiny 7};
\draw [very thin] (a) -- (e);
\draw [very thin] (e) -- (g);
\draw [very thin] (e) -- (b);
\draw [very thin] (e) -- (c);
\draw [red,snake=zigzag,segment amplitude=1pt, very thin] (f) -- (d);
\draw [very thin] (g) -- (c);
\draw [very thin] (g) -- (d);
\draw [very thin] (f) -- (g);
\end{scope}
\end{tikzpicture}
 \\

\begin{tikzpicture}
\begin{scope}
\node at (.2,0) {$\scriptstyle d_8=2$};
\draw [->,thin] (0,-.3) -- (.4,-.3);
\node at (1,.5) [shape=circle,draw, inner sep=0pt,minimum size=3mm,fill=black!15] (a) {\tiny 1};
\node at (1,-.7) [shape=circle,draw, inner sep=0pt,minimum size=3mm,fill=black!15] (b) {\tiny 2};
\node at (1.7,-.5) [shape=circle,draw, inner sep=0pt,minimum size=3mm,fill=black!15] (c) {\tiny 3};
\node at (2.2,.5) [shape=circle,draw, inner sep=0pt,minimum size=3mm,fill=black!15] (d) {\tiny 4};
\node at (1.35,0) [shape=circle,draw, inner sep=0pt,minimum size=3mm,fill=black!15] (e) {\tiny 5};
\node at (2.6,.-.4) [shape=circle,draw, inner sep=0pt,minimum size=3mm,fill=black!15] (f) {\tiny 6};
\node at (2,0) [shape=circle,draw, inner sep=0pt,minimum size=3mm,fill=black!15] (g) {\tiny 7};
\node at (3,0) [shape=circle,draw, inner sep=0pt,minimum size=3mm,fill=black!15] (h) {\tiny 8};
\draw [very thin] (a) -- (e);
\draw [very thin] (e) -- (g);
\draw [very thin] (e) -- (b);
\draw [very thin] (e) -- (c);
\draw [very thin] (f) -- (h);
\draw [red,snake=zigzag,segment amplitude=1pt, very thin] (d) -- (h);
\draw [red,snake=zigzag,segment amplitude=1pt, very thin] (g) -- (c);
\draw [very thin] (g) -- (d);
\draw [very thin] (f) -- (g);
\end{scope}

\begin{scope}[shift={(4,0)}]
\node at (.2,0) {$\scriptstyle d_9=4$};
\draw [->,thin] (0,-.3) -- (.4,-.3);
\node at (1,.5) [shape=circle,draw, inner sep=0pt,minimum size=3mm,fill=black!15] (a) {\tiny 1};
\node at (1,-.7) [shape=circle,draw, inner sep=0pt,minimum size=3mm,fill=black!15] (b) {\tiny 2};
\node at (1.7,-.8) [shape=circle,draw, inner sep=0pt,minimum size=3mm,fill=black!15] (c) {\tiny 3};
\node at (2.2,.5) [shape=circle,draw, inner sep=0pt,minimum size=3mm,fill=black!15] (d) {\tiny 4};
\node at (1.35,0) [shape=circle,draw, inner sep=0pt,minimum size=3mm,fill=black!15] (e) {\tiny 5};
\node at (3.3,.-.6) [shape=circle,draw, inner sep=0pt,minimum size=3mm,fill=black!15] (f) {\tiny 6};
\node at (2,0) [shape=circle,draw, inner sep=0pt,minimum size=3mm,fill=black!15] (g) {\tiny 7};
\node at (3,0) [shape=circle,draw, inner sep=0pt,minimum size=3mm,fill=black!15] (h) {\tiny 8};
\node at (2.5,-.8) [shape=circle,draw, inner sep=0pt,minimum size=3mm,fill=black!15] (i) {\tiny 9};
\draw [very thin] (a) -- (e);
\draw [very thin] (e) -- (g);
\draw [red,snake=zigzag,segment amplitude=1pt, very thin] (e) -- (b);
\draw [very thin] (e) -- (c);
\draw [very thin] (f) -- (h);
\draw [red,snake=zigzag,segment amplitude=1pt, very thin] (i) -- (h);
\draw [very thin] (i) -- (d);
\draw [very thin] (g) -- (i);
\draw [very thin] (i) -- (c);
\draw [red,snake=zigzag,segment amplitude=1pt, very thin] (g) -- (d);
\draw [very thin] (f) -- (g);
\end{scope}

\begin{scope}[shift={(9,0)}]
\node at (.2,0) {$\scriptstyle d_{10}=6$};
\draw [->,thin] (0,-.3) -- (.4,-.3);
\node at (1,.5) [shape=circle,draw, inner sep=0pt,minimum size=3mm,fill=black!15] (a) {\tiny 1};
\node at (1,-.7) [shape=circle,draw, inner sep=0pt,minimum size=3mm,fill=black!15] (b) {\tiny 2};
\node at (1.7,-.8) [shape=circle,draw, inner sep=0pt,minimum size=3mm,fill=black!15] (c) {\tiny 3};
\node at (2.5,.5) [shape=circle,draw, inner sep=0pt,minimum size=3mm,fill=black!15] (d) {\tiny 4};
\node at (1.35,0) [shape=circle,draw, inner sep=0pt,minimum size=3mm,fill=black!15] (e) {\tiny 5};
\node at (3.3,.-.8) [shape=circle,draw, inner sep=0pt,minimum size=3mm,fill=black!15] (f) {\tiny 6};
\node at (2,0) [shape=circle,draw, inner sep=0pt,minimum size=3mm,fill=black!15] (g) {\tiny 7};
\node at (3.9,-.2) [shape=circle,draw, inner sep=0pt,minimum size=3mm,fill=black!15] (h) {\tiny 8};
\node at (2.5,-.8) [shape=circle,draw, inner sep=0pt,minimum size=3mm,fill=black!15] (i) {\tiny 9};
\node at (3.3,.2) [shape=circle,draw, inner sep=0pt,minimum size=3mm,fill=black!15] (j) {\tiny 10};
\draw [very thin] (a) -- (e);
\draw [very thin] (e) -- (g);
\draw [very thin] (e) .. controls (1,1) and (2.5,1.2) .. (j);
\draw [very thin] (j) -- (b);
\draw [very thin] (e) -- (c);
\draw [very thin] (f) -- (h);
\draw [very thin] (j) -- (h);
\draw [red,snake=zigzag,segment amplitude=1pt, very thin] (i) -- (j);
\draw [very thin] (i) -- (d);
\draw [very thin] (g) -- (i);
\draw [very thin] (i) -- (c);
\draw [very thin] (j) -- (d);
\draw [very thin] (g) -- (j);
\draw [very thin] (f) -- (g);
\end{scope}
\end{tikzpicture}

\bigskip

The proof that the DPG-process creates an infinite sequence of prime gap graphs incorporates two main ingredients. The first one is a symmetric inequality, which implies the
Erd\H{o}s--Gallai inequalities, and thus provides a practical sufficient condition
for the graphicality of the underlying degree sequence. Moreover, another new inequality implies the DPG-graphicality of a
long sequence of natural numbers. The second ingredient combines classical $L^2$ and $L^\infty$ bounds for prime gaps.
\begin{theorem}\label{thm:prime-gap-graphicness}
If $n$ is sufficiently large, then the initial segment $\PD^n$ is graphic, and for any realization $G_n$ of $\PD^n$, the pair $(G_n, p_{n+1}-p_n)$ is DPG-graphic.
\end{theorem}
Ultimately, the proof relies on the rarity of zeros possibly violating the Riemann hypothesis. In principle, it allows one to deduce an effective constant beyond which \Cref{thm:prime-gap-graphicness} holds true, but this constant far exceeds the capabilities of computers. However, assuming the Riemann hypothesis, we can reduce the constant significantly and prove the result for all $n\geq 2$. Here the numerics are quite delicate, and for efficiency we depart from the symmetric treatment alluded to above. Instead, we go back to first principles and examine the contribution of large prime gaps more directly. We still need to rely on computational results, but they can either be obtained with very simple computer programs, or found in the literature (e.g.\ prime gap records until $2\cdot 10^{18}$).
\begin{theorem}\label{thm:prime-gap-graphicness-under-RH}
Assume the Riemann hypothesis. Then, for any $n\geq 2$, the initial segment $\PD^n$ is graphic, and for any realization $G_n$ of $\PD^n$, the pair $(G_n, p_{n+1}-p_n)$ is DPG-graphic.
\end{theorem}
Note that Theorems~\ref{thm:prime-gap-graphicness} and \ref{thm:prime-gap-graphicness-under-RH} are a reformulation of \Cref{maintheorem} in the terminology of combinatorics and network theory. They rely on the core theorems presented in the next subsection, which are of independent interest.

\medskip

\subsection{New results}\label{newresults}

Our first result provides, via two symmetric inequalities, sufficient conditions for 
a given sequence to be graphic and that in every graph realization of the sequence,
there is a matching of a given size.

\begin{theorem}\label{thm1}
Let $\D=(d_1,\dotsc,d_n)$ be a sequence of positive integers such that ${\|\D\|}_1=\sum_{\ell=1}^n d_\ell$ is even.
Let $1<p\leq\infty$ be a parameter.\\
{\normalfont\textbf{Part (a).}} Assume that the following $L^p$-norm bound holds:
\begin{equation}\label{eq:p-norm-1}
{\|2+\D\|}_p\leq n^{\frac{1}{2}+\frac{1}{2p}}.
\end{equation}
Then there is a simple graph $G$ with degree sequence $\D$.\\
{\normalfont\textbf{Part (b).}} Let $G$ be any simple graph with degree sequence $\D$. Assume that $d\geq 2$ is an even integer satisfying
\begin{equation}\label{eq:p-norm-2}
4d^{1-\frac{1}{p}}{\|\D\|}_p\leq{\|\D\|}_1.
\end{equation}
Then the pair $(G,d)$ is DPG-graphic, and consequently $\D\circ d$ is graphic.
\end{theorem}

Our second result makes explicit a theorem of Selberg~\cite[Th.~2]{Se}.

\begin{theorem}\label{explicitselberg} Assume the Riemann hypothesis. Then, for any $x\geq 2$ and $N>0$, we have
\begin{equation}\label{HBeffective}
\sum_{\substack{x\leq p_\ell\leq 2x\\p_{\ell+1}-p_\ell\geq N}} (p_{\ell+1}-p_\ell) < \frac{163x\log^2 x}{N}.
\end{equation}
\end{theorem}

\begin{remark} The example $x=2$ and $N=2$ shows that this result would become false if we replaced
the constant $163$ by $4$. On the other hand, Cram\'er's model predicts that it can be replaced by $o(1)$ for $x\to\infty$ (cf.\ \cite[\S 1.1]{FGKMT}).
\end{remark}

In order to achieve the good numeric constant $163$, we estimate carefully (among others) the error term in the truncated von Mangoldt formula for the Chebyshev psi function. This result, stated below, makes explicit a theorem of Goldston~\cite{Go}, and simultaneously extends and sharpens a theorem of Dudek~\cite[Th.~1.3]{Du} in the special case relevant for us. Here and later the notation $A=O^*(B)$ stands for $|A|\leq B$.

\begin{theorem}\label{dudeksharpened} For any $z>x>10^{18}$ we have
\[\psi(x) = x-\sum_{|\Im\rho|<z}\frac{x^{\rho}}{\rho} + O^*\bigl(5\log x\log\log x\bigr),\]
where the sum is over the nontrivial zeros of the Riemann zeta function (counted with multiplicity).
\end{theorem}

\section{Preliminary results}

\subsection{An application of Vizing's theorem}

\begin{theorem}[Vizing~\cite{vizing}]\label{thmvizing} A simple graph with maximal degree $\Delta$ admits a proper edge coloring with $\Delta+1$ colors.
\end{theorem}

\begin{lemma}\label{thm2} Let $G$ be a simple graph on $n$ vertices with degrees $d_1,\dotsc,d_n$. Let $\delta\geq 1$ be an integer, and let $d\geq 2$ be an even integer satisfying
\begin{equation}\label{eq:DPG-good}
\delta d\leq\sum_{d_\ell<\delta}d_\ell-\sum_{d_\ell\geq\delta}d_\ell.
\end{equation}
Then $G$ has a matching of size $d/2$.
\end{lemma}

\begin{proof} Let us delete all vertices of degree at least $\delta$ (and the incident edges) from $G$. The remaining subgraph $H$ has maximal degree less than $\delta$, and number of edges at least
\begin{equation}\label{numberofedges}
\frac{1}{2}\sum_{\ell=1}^n d_\ell-\sum_{d_\ell\geq\delta}d_\ell =
\frac{1}{2}\left(\sum_{d_\ell<\delta}d_\ell-\sum_{d_\ell\geq\delta}d_\ell\right)
\geq\frac{\delta d}{2}
\end{equation}
by \eqref{eq:DPG-good}. It follows from \Cref{thmvizing} that the edge set of $H$ can be partitioned into $\delta$ matchings, and
then \eqref{numberofedges} shows that the largest matching in this decomposition must be of size at least $d/2$. Since $H$ is a subgraph of $G$, the proof is complete.
\end{proof}

\subsection{Preliminaries about $\Gamma(z)$ and $\zeta(s)$}

\begin{lemma}\label{gammabound}
Assume that $\Re z>0$. Then
\[\Re\frac{\Gamma'(z)}{\Gamma(z)}+\frac{\Gamma'(\Re z)}{\Gamma(\Re z)}<
\log|z|+\log\Re z-\Re\frac{1}{2z}-\frac{1}{2\Re z}.\]
\end{lemma}

\begin{proof} With the help of the well-known integral representation (cf.\ \cite[\S 12.31]{WW})
\begin{equation}\label{gammaintegral}
\frac{\Gamma'(z)}{\Gamma(z)}=\log z-\frac{1}{2z}
-\int_0^\infty\left(\frac{1}{2}-\frac{1}{t}+\frac{1}{e^t-1}\right)e^{-tz}\,dt,\qquad\Re z >0,
\end{equation}
the statement becomes
\[\int_0^\infty\left(\frac{1}{2}-\frac{1}{t}+\frac{1}{e^t-1}\right)\left(\Re e^{-tz}+e^{-t\Re z}\right)\,dt>0.\]
However, this one is clear, because the integrand is non-negative with a discrete set of zeros (there are no zeros when $z$ is real).
\end{proof}

\begin{lemma}[Delange~\cite{De}]\label{zetalemma1} For any $\sigma>1$ and $t\in\RR$ we have
\[\left|\frac{\zeta'(\sigma+it)}{\zeta(\sigma+it)}\right|<\frac{1}{\sigma-1}-\frac{1}{2\sigma^2}.\]
\end{lemma}

\begin{lemma}[Dudek~\cite{Du}]\label{zetalemma2}
Let $\sigma\leq -1$ and $t\in\RR$. Assume that either $\sigma\in 1+2\ZZ$ or $|t|\geq 1$. Then
\[\left|\frac{\zeta'(\sigma+it)}{\zeta(\sigma+it)}\right|<9+\log|\sigma+it|.\]
\end{lemma}

\begin{proof} This is a variant of \cite[Lem.~2.3]{Du}, and can be proved in the same way. We note a small oversight in \cite[p.~183]{Du}: instead of assuming that $U\geq 2$ is an even integer, one should assume that $U\geq 1$ is an odd integer, just as in \cite[\S 17]{Da}.
\end{proof}

\begin{lemma}[Dudek~\cite{Du}]\label{zetalemma3}
Assume that $z>100$. Then there exists $T\in(z-2,z)$ such that
\begin{equation}\label{eq:log-der-hline-high}
\left|\frac{\zeta'(\sigma+iT)}{\zeta(\sigma+iT)}\right|<\log^2 z + 20\log z,\qquad \sigma>-1.
\end{equation}
\end{lemma}

\begin{proof} The statement follows from \cite[Lem.~2.8]{Du}.
\end{proof}

\begin{definition}
For $T>0$, we denote by $N(T)$ the number of zeros of $\zeta(s)$ with imaginary part in $(0,T)$, counted with multiplicity.
\end{definition}

\begin{lemma}\label{deltazerocount} For any $T\geq\Delta+2\pi>\Delta>0$ we have
\[N(T)-N(T-\Delta)<\left(\frac{\Delta}{2\pi}+0.56\right)\log T.\]
\end{lemma}

\begin{proof} By \cite[Cor.~1]{BPT}, we have
\[N(T)-N(T-\Delta)=\frac{1}{2\pi}\int_{T-\Delta}^T \log\frac{t}{2\pi}\, dt +O^*(0.56\log T).\]
The integrand is less than $\log T$, and the result follows.
\end{proof}

\subsection{Preliminaries about prime gaps}

\begin{theorem}[Ingham {\cite[Th.~4]{I1}}]\label{inghamtheorem}
Let $\eps>0$. For any $x\geq x_0(\eps)$, there is a prime number in $[x,x+x^{5/8+\eps}]$.
\end{theorem}

\begin{remark} The exponent $5/8+\eps$ was improved multiple times, the current record $21/40$ being due to Baker--Harman--Pintz~\cite{BHP}. We have emphasized the classical result of Ingham~\cite{I1} as it suffices for our purposes.
\end{remark}

\begin{theorem}[Carneiro--Milinovich--Soundararajan {\cite[Th.~1.5]{CMS}}]\label{cmstheorem}
Assume the Riemann hypothesis. Then, for any $x\geq 4$, there is a prime number in $[x,x+\frac{22}{25}\sqrt{x}\log x]$.
\end{theorem}

In a restricted range, we have a stronger unconditional result thanks to explicit calculations.

\begin{lemma}\label{legendrevariant}
For any $x\in[117,10^{18}]$, there is a prime number in $[x,x+\sqrt{x}]$.
\end{lemma}

\begin{proof} Assume that the conclusion fails for some $x\in[117,10^{18}]$.
    Then there is a unique prime number $p_\ell$ such that
    $p_\ell<x<x+\sqrt{x}<p_{\ell+1}$. In particular, $\ell\geq 31$ and
    $p_\ell<x<{(p_{\ell+1}-p_\ell)}^2$. Hence the computations of Oliveira e
    Silva, Herzog, and Pardi~\cite[Table~8]{OHP} show that the initial upper
    bound ${10}^{18}$ for $p_\ell$ successively improves to: $1442^2$, $148^2$,
    $52^2$, $34^2$, $22^2$, $14^2$. This means that $31\leq \ell\leq 44$, but
    then $x<{(p_{\ell+1}-p_\ell)}^2\leq 10^2$ is a contradiction.
\end{proof}

\begin{remark} The conclusion of \Cref{legendrevariant} is likely true for all $x\geq 117$. However, this statement is not known to follow from the Riemann hypothesis, and it is stronger than Oppermann's conjecture (which itself implies Legendre's conjecture, Andrica's conjecture, and Brocard's conjecture).
\end{remark}

\begin{theorem}[Heath-Brown {\cite{HB1}}]\label{hbtheorem}
For any $x\geq 2$ we have
\[\sum_{p_\ell\leq x}(p_{\ell+1}-p_\ell)^2\ll x^{4/3}(\log x)^{10000}.\]
\end{theorem}

\begin{remark} The exponent $4/3$ was improved to $23/18+\eps$ by Heath-Brown~\cite{HB2}, and to $5/4+\eps$ independently by Peck~\cite{Pe} and Maynard~\cite{Ma}. The current record $123/100+\eps$ is due to Stadlmann~\cite{St}. We have emphasized the original breakthrough of Heath-Brown~\cite{HB1} as it suffices for our purposes.
\end{remark}

\section{Proof of the main theorem}

In this section, we first prove \Cref{thm:prime-gap-graphicness} assuming \Cref{thm1}, and then we prove \Cref{thm:prime-gap-graphicness-under-RH} assuming \Cref{explicitselberg}. In other words, we deduce \Cref{maintheorem} from Theorems~\ref{thm1} and \ref{explicitselberg}.

\subsection{Proof of \Cref{thm:prime-gap-graphicness}}

Let $n$ be sufficiently large. We shall verify the conditions of \Cref{thm1} for
\[p:=2,\qquad d_\ell:=p_\ell-p_{\ell-1},\qquad d:=p_{n+1}-p_n.\]
Clearly,
${\|\D\|}_1=p_n-1$ is even. Condition \eqref{eq:p-norm-1} reads
\begin{equation}\label{eq:p-norm-3}
\sum_{\ell=1}^n (2+p_\ell-p_{\ell-1})^2\leq n^{3/2},
\end{equation}
while condition \eqref{eq:p-norm-2} reads
\begin{equation}\label{eq:p-norm-4}
16(p_{n+1}-p_n)\sum_{\ell=1}^n (p_\ell-p_{\ell-1})^2\leq(p_n-1)^2.
\end{equation}
Now \eqref{eq:p-norm-3} and \eqref{eq:p-norm-4} follow from
Theorems~\ref{inghamtheorem} and \ref{hbtheorem}, hence we are done:
\begin{gather*}
\sum_{\ell=1}^n (2+p_\ell-p_{\ell-1})^2\leq 9\sum_{\ell=1}^n (p_\ell-p_{\ell-1})^2\leq p_n^{4/3+o(1)}=n^{4/3+o(1)},\\
16(p_{n+1}-p_n)\sum_{\ell=1}^n (p_\ell-p_{\ell-1})^2\leq p_n^{5/8+o(1)}p_n^{4/3+o(1)}=p_n^{47/24+o(1)}.
\end{gather*}

\subsection{Proof of \Cref{thm:prime-gap-graphicness-under-RH}}

Assume the Riemann hypothesis, and let $G$ be a prime gap graph on $n$ vertices. It suffices to show that $G$ has $(p_{n+1}-p_n)/2$ independent edges (cf.\ \Cref{conj2}), because then a straightforward induction argument based on \Cref{tm:weak} shows that
every initial segment of $\PD$ is graphic (cf.\ \Cref{conj1}). The statement is clear for $2\leq n\leq 4$, hence we shall restrict to $n\geq 5$.

\medskip

By \Cref{thm2}, it suffices to exhibit an integer $N\geq 1$ satisfying
\begin{equation}\label{delta}
N(p_{n+1}-p_n)+2\!\!\!\!\!\!\sum_{\substack{1\leq\ell\leq n\\p_\ell-p_{\ell-1}\geq N}}(p_\ell-p_{\ell-1})< p_n.
\end{equation}
For $p_n<10^{18}$ we take
\[N:=\max_{1\leq\ell\leq n}(1+p_\ell-p_{\ell-1}),\]
so that \eqref{delta} simplifies to $N(p_{n+1}-p_n)<p_n$. In fact the proof of \Cref{thm2} reveals that the last condition can be relaxed to
\begin{equation}\label{delta2}
\frac{p_{n+1}-p_n}{2}\leq\left\lceil\frac{p_n-1}{2 N}\right\rceil,
\end{equation}
which works better for very small $n\geq 5$. For $5\leq n\leq 44$ the condition \eqref{delta2} can be checked by a simple computer program (or by hand). For $n\geq 45$ and $p_n<10^{18}$ we verify \eqref{delta2} as follows. Let $k$ be the unique positive integer satisfying $(k-1)^2<p_n<k^2$. Note that $k\geq 15$, because $p_n\geq p_{45}=197$. From \Cref{legendrevariant} it follows that $p_{n+1}-p_n\leq k-1$ and
\[N=\max\left(15,\max_{32\leq\ell\leq n}(1+p_\ell-p_{\ell-1})\right)\leq k,\]
hence also that
\[\frac{p_n-1}{2N}>\frac{k^2-2k}{2N}\geq\frac{k^2-2k}{2k}=\frac{k}{2}-1.\]
Therefore, \eqref{delta2} is clear by
\[\frac{p_{n+1}-p_n}{2}\leq\frac{k-1}{2}\leq\left\lceil\frac{p_n-1}{2N}\right\rceil.\]

For $p_n>10^{18}$ we take
\[N:=\left\lceil\frac{\sqrt{p_n}}{3\log p_n}\right\rceil.\]
Then, by Theorems~\ref{cmstheorem} and \ref{explicitselberg}, we have
\[N(p_{n+1}-p_n)<\frac{p_n}{3}\qquad\text{and}\qquad
\sum_{\substack{1\leq\ell\leq n\\p_\ell-p_{\ell-1}\geq N}}(p_\ell-p_{\ell-1})<489\sqrt{p_n}\log^3 p_n<\frac{p_n}{3}.\]
From here the bound \eqref{delta} is immediate, hence we are done.

\section{A symmetric criterion for graphicality}\label{sec:gen}

In this section, we prove \Cref{thm1}.

\medskip

\textbf{Part (a).} By symmetry, we can assume that $d_1\geq\dotsb\geq d_n$. By \Cref{tm:EG}, it suffices to check that for any $1\leq k\leq n$,
\[
\sum_{\ell=1}^k d_\ell \leq k(k-1) + \sum_{\ell=k+1}^n \min(k,d_\ell).
\]
Since $d_\ell\geq 1$ for any $1\leq\ell\leq n$, it suffices to prove that
\[
\sum_{\ell=1}^k d_\ell \leq k(k-1) + (n-k),
\]
which is equivalent to
\[
\sum_{\ell=1}^k (2+d_\ell) \leq k^2+n.
\]
This last condition follows from \eqref{eq:p-norm-1} and H\"older's inequality, hence we are done:
\[{\|2+\D^k\|}_1 \leq k^{1-\frac{1}{p}}{\|2+\D^k\|}_p\leq k^{1-\frac{1}{p}}{\|2+\D\|}_p
\leq k^{1-\frac{1}{p}} n^{\frac{1}{2}+\frac{1}{2p}}\leq\max(k^2,n).\]
In the last step, we used that both $k^2$ and $n$ are upper bounded by $\max(k^2,n)$.

\medskip

\textbf{Part (b).} By \Cref{tm:weak} and \Cref{thm2}, it suffices to verify that \eqref{eq:DPG-good} holds for some integer $\delta\geq 1$. If $p=\infty$, then \eqref{eq:p-norm-2} says that $4d{\|\D\|}_\infty\leq{\|\D\|}_1$, hence \eqref{eq:DPG-good} holds for $\delta:=1+{\|\D\|}_\infty$. So let us focus on the case $1<p<\infty$. For any integer $\delta\geq 1$, we have
\[{\|\D\|}_p^p\geq\sum_{d_\ell\geq\delta}d_\ell^p\geq\delta^{p-1}\sum_{d_\ell\geq\delta}d_\ell,\]
hence also
\[\sum_{d_\ell<\delta}d_\ell-\sum_{d_\ell\geq\delta}d_\ell\geq{\|\D\|}_1-2\delta^{1-p}{\|\D\|}_p^p.\]
So for the validity of \eqref{eq:DPG-good}, it suffices that
\[\delta^{1-p}{\|\D\|}_p^p\leq\frac{1}{4}{\|\D\|}_1\qquad\text{and}\qquad
\delta d\leq\frac{1}{2}{\|\D\|}_1.\]
In other words, it suffices to find an integer $\delta$ satisfying
\[\left(\frac{4{\|\D\|}_p^p}{{\|\D\|}_1}\right)^\frac{1}{p-1}\leq\delta\leq\frac{1}{2d}{\|\D\|}_1.\]
The left-hand side exceeds $1$, hence $\delta$ exists as long as
\[2\left(\frac{4{\|\D\|}_p^p}{{\|\D\|}_1}\right)^\frac{1}{p-1}\leq\frac{1}{2d}{\|\D\|}_1.\]
This is equivalent to condition \eqref{eq:p-norm-2}, hence the proof of \Cref{thm1} is complete.

\section{The sum of large prime gaps}

In this section, we prove \Cref{explicitselberg} assuming \Cref{dudeksharpened}. Throughout, we assume the Riemann hypothesis.

\medskip

First we eliminate some simple cases. Let $N^*$ denote the largest prime gap $p_{\ell+1}-p_\ell$ occurring in \eqref{HBeffective}. Then we can clearly assume that
\begin{equation}\label{assumption1}
N\leq N^*\qquad\text{and}\qquad 2x+N^*>\frac{163x\log^2 x}{N},
\end{equation}
hence also that $N^*(2x+N^*)>163x\log^2 x$. From \Cref{cmstheorem} it follows that $N^*<3\sqrt{x}\log x$, so our previous inequality yields $N^*>77\log^2 x$. By \cite[Table~8]{OHP}, this forces $x>10^{18}$. Indeed,
for $x\in[2,10^3]$ we have $N^*\leq 34$, while for $x\in[10^3,10^{18}]$ we have $N^*\leq 1476$. On the other hand, for $x>10^{18}$ we get from \Cref{cmstheorem} that $N^*<\frac{4}{3}\sqrt{x}\log x<0.001 x$, hence by \eqref{assumption1} also that
\begin{equation}\label{assumption2}
81\log^2 x<N<\frac{4}{3}\sqrt{x}\log x.
\end{equation}

From now on we assume both $x>10^{18}$ and \eqref{assumption2}. Following Heath-Brown~\cite{HB1}, we write
$N=4\delta x$ with
\[\frac{81\log^2 x}{4x}<\delta<\frac{\log x}{3\sqrt{x}},\]
and we set out to estimate the square mean of
\[E(y,\delta) := \psi(y+\delta y)-\psi(y)-\delta y,\qquad x\leq y\leq 2x.\]
It follows from \Cref{dudeksharpened} and the crude bound $y+\delta y<3x$ that
\[
|E(y,\delta)|<\left|\sum_{|\Im\rho|<3x} y^{\rho}C(\rho)\right| + \log^2 x,
\]
where
\[
C(\rho):=\frac{1-(1+\delta)^{\rho}}{\rho}.
\]
As a result,
\begin{equation}\label{eq:heath-brown-sum-over-zeros-CS}
\int_x^{2x} |E(y,\delta)|^2\, dy < 2\int_x^{2x}\left|\sum_{|\Im\rho|<3x} y^{\rho}C(\rho)\right|^2\,dy+2x\log^4 x.
\end{equation}
\begin{lemma}\label{Cbound} We have
\begin{equation}\label{eq:c-rho-bound-b}
|C(\rho)|<\min\left(\delta,\frac{\sqrt{4+4\delta}}{|\rho|}\right).
\end{equation}
\end{lemma}
\begin{proof}
The bound $|C(\rho)|<\delta$ is a consequence of the integral representation
\[C(\rho)=\int_{1+\delta}^1 x^{\rho-1}\,dx\]
and the triangle inequality for complex-valued Riemann integrals. In addition, the triangle inequality for complex numbers yields by the definition of $C(\rho)$ that
\[|C(\rho)|\leq\frac{1+\sqrt{1+\delta}}{|\rho|}<\frac{\sqrt{4+4\delta}}{|\rho|}.\qedhere\]
\end{proof}

We estimate the integral on the right-hand side of \eqref{eq:heath-brown-sum-over-zeros-CS} as in the proof of \cite[Lem.~5]{SV}:
\begin{align*}
\int_x^{2x}\left|\sum_{|\Im\rho|<3x} y^{\rho}C(\rho)\right|^2\,dy
&\leq\int_1^2\int_{xv/2}^{2xv}\left|\sum_{|\Im\rho|<3x} y^{\rho}C(\rho)\right|^2\,dy\,dv\\
&=\sum_{|\Im\rho|,|\Im\rho'|<3x}x^{2+\rho-\rho'}C(\rho)\ov{C(\rho')}\cdot
\frac{2^{2+\rho-\rho'}-2^{-2-\rho+\rho'}}{2+\rho-\rho'}\cdot
\frac{2^{3+\rho-\rho'}-1}{3+\rho-\rho'}\\
&\leq x^2\sum_{|\Im\rho|,|\Im\rho'|<3x}|C(\rho)C(\rho')|
\left|\frac{2^2+2^{-2}}{2+\rho-\rho'}\right|\left|\frac{2^3+1}{3+\rho-\rho'}\right|\\
&\leq x^2\sum_{|\Im\rho|,|\Im\rho'|<3x}|C(\rho)|^2
\left|\frac{2^2+2^{-2}}{2+\rho-\rho'}\right|\left|\frac{2^3+1}{3+\rho-\rho'}\right|.
\end{align*}
Here the contribution of $\Im\rho<0$ is the same as the contribution of $\Im\rho>0$. Therefore, applying \Cref{Cbound} along with the elementary inequality
\[|2+it|\cdot|3+it|\geq 6+t^2,\qquad t\in\RR,\]
we arrive at
\begin{equation}\label{pairsum}
\int_x^{2x}\left|\sum_{|\Im\rho|<3x} y^{\rho}C(\rho)\right|^2\,dy
< \frac{153}{2} x^2\sum_{0<\Im\rho<3x}\min\left(\delta^2,\frac{4+4\delta}{|\rho|^2}\right)
\sum_{|\Im\rho'|<3x}\frac{1}{6+|\rho-\rho'|^2}.
\end{equation}

We need to estimate the inner sum in \eqref{pairsum}. The idea is to drop the condition $|\Im\rho'|<3x$, and consider the full convergent series
\[\sum_{\rho'}\frac{1}{6+|\rho-\rho'|^2}
=\frac{1}{\sqrt{6}}\sum_{\rho'}\Re\frac{1}{\sqrt{6}+\rho-\rho'}
=\frac{1}{\sqrt{6}}\Re\frac{\xi'(\sqrt{6}+\rho)}{\xi(\sqrt{6}+\rho)}.\]
The first equality follows from the Riemann hypothesis, while the second equality follows from \cite[Cor.~10.14]{MV}.
Let us write $s:=\sqrt{6}+\rho$ for simplicity. Then, as in the proof of \cite[Cor.~10.14]{MV}, we have
\[\frac{\xi'(s)}{\xi(s)}=-\frac{1}{2}\log\pi+\frac{1}{s-1}+\frac{\zeta'(s)}{\zeta(s)}+\frac{1}{2}\frac{\Gamma'(s/2+1)}{\Gamma(s/2+1)}.\]
We take the real part of both sides, and apply \Cref{gammabound}:
\begin{align*}
\Re\frac{\xi'(s)}{\xi(s)}
<&-\frac{1}{2}\log\pi+\Re\frac{1}{s-1}-\frac{\zeta'(\Re s)}{\zeta(\Re s)}
+\frac{1}{2}\Re\frac{\Gamma'(s/2+1)}{\Gamma(s/2+1)}\\
<&-\frac{1}{2}\log\pi+\Re\frac{1}{s-1}-\Re\frac{1}{2s+4}-\frac{\zeta'(\Re s)}{\zeta(\Re s)}-\frac{1}{2\Re s+4}\\
&-\frac{1}{2}\frac{\Gamma'(\Re s/2+1)}{\Gamma(\Re s/2+1)}
+\frac{1}{2}\log(\Re s/2+1)+\frac{1}{2}\log|s/2+1|.
\end{align*}
Using that $\Re s=\sqrt{6}+1/2$ and $\Im s>14$, it is straightforward to check that $\Re\frac{1}{s-1}<\Re\frac{1}{2s+4}$, hence in fact
\[\Re\frac{\xi'(s)}{\xi(s)}
<0.181-\frac{1}{2}\log\pi+\frac{1}{2}\log|s/2+1|
<\frac{1}{4}+\frac{1}{2}\log\frac{\Im\rho}{2\pi}.\]
To sum up, we have proved that
\[\sum_{\rho'}\frac{1}{6+|\rho-\rho'|^2}<\frac{1}{2\sqrt{6}}\left(\frac{1}{2}+\log\frac{\Im\rho}{2\pi}\right).\]

Going back to \eqref{pairsum}, we conclude that
\[\int_x^{2x}\left|\sum_{|\Im\rho|<3x} y^{\rho}C(\rho)\right|^2\,dy
<15.616 x^2\sum_{0<\Im\rho<3x}\min\left(\delta^2,\frac{4}{|\rho|^2}\right)
\left(\frac{1}{2}+\log\frac{\Im\rho}{2\pi}\right).\]
In the sum on the right-hand side, we drop the condition $\Im\rho<3x$ and replace $|\rho|$ by $\Im\rho$. By \cite[Cor.~1, Lem.~5--6]{BPT}, the resulting bigger sum can be estimated as follows:
\begin{align*}
&<\delta^2N(2/\delta)\left(\frac{1}{2}+\log\frac{1}{\delta\pi}\right)+
\sum_{\Im\rho\geq 2/\delta}\left(\frac{2}{(\Im\rho)^2}+\frac{4}{(\Im\rho)^2}\log\frac{\Im\rho}{2\pi}\right)\\
&<\frac{\delta}{\pi}\left(\log\frac{1}{\delta\pi}\right)\left(\frac{1}{2}+\log\frac{1}{\delta\pi}\right)+
\frac{\delta}{2\pi}\log\frac{2}{\delta}+
\frac{\delta}{\pi}\left(\log^2\frac{2}{\delta}-\log\frac{2}{\delta}\right)\\
&<\frac{\delta}{\pi}\left(\log^2\frac{1}{\delta\pi}+\log^2\frac{2}{\delta}\right)
<\frac{2\delta}{\pi}\log^2 x.
\end{align*}
In the end, we get
\[\int_x^{2x}\left|\sum_{|\Im\rho|<3x} y^{\rho}C(\rho)\right|^2\,dy
<9.942\delta x^2\log^2 x.\]
Plugging this bound into \eqref{eq:heath-brown-sum-over-zeros-CS}, we conclude that
\begin{equation}
\int_x^{2x}\label{eq:heath-brown-lemma-explicit}
|E(y,\delta)|^2\, dy < 19.884\delta x^2\log^2 x + 2x\log^4 x < 19.983\delta x^2\log^2 x.
\end{equation}

Assume now that the prime $p_\ell\in [x,2x]$ satisfies $p_{\ell+1}-p_\ell\geq N$. There is at most one $p_\ell$ such that
$(p_\ell+p_{\ell+1})/2>2x$, so assume also that $(p_\ell+p_{\ell+1})/2\leq 2x$. Then, for any
\[y\in(p_\ell,(p_\ell+p_{\ell+1})/2)\subset(x,2x),\]
the interval
\[[y,y+\delta y]\subset[y,y+N/2]\subset(p_\ell,p_{\ell+1})\]
is free of primes, hence counting the possible higher prime powers in this interval, we get
\[\psi(y+\delta y)-\psi(y)\leq(1+\delta\sqrt{y}/2)\log_2(y+\delta y)<(2+\delta\sqrt{y})\log x<0.003\delta y.\]
That is, $|E(y,\delta)|>0.997\delta y$ holds on $(p_\ell,(p_\ell+p_{\ell+1})/2)$. Squaring and integrating, we get
\[\int_{p_\ell}^{(p_\ell+p_{\ell+1})/2}|E(y,\delta)|^2\,dy > 0.497\delta^2 x^2(p_{\ell+1}-p_\ell).\]
Summing over all such primes $p_\ell$, and using \eqref{eq:heath-brown-lemma-explicit} as well as \Cref{cmstheorem} for the possible single exceptional $p_\ell$, we obtain \eqref{HBeffective}:
\[\sum_{\substack{x\leq p_\ell\leq 2x\\p_{\ell+1}-p_\ell\geq N}} (p_{\ell+1}-p_\ell)
< \frac{4}{3}\sqrt{x}\log x+\frac{1}{0.497\delta^2 x^2}\int_x^{2x} |E(y,\delta)|^2\,dy < \frac{163x\log^2 x}{N}.\]

\section{The error term in the truncated von Mangoldt formula}

In this section, we prove \Cref{dudeksharpened}. We follow Davenport~\cite[\S 17]{Da} and Goldston~\cite{Go} with appropriate modifications.

\medskip

We assume first that $x\not\in\ZZ$. We choose $T\in(z-2,z)$ according to \Cref{zetalemma3}, and we also fix $c:=1+1/\log x$. We record the following approximation to the characteristic function of $(1,\infty)$:
\[\mathbf{1}_{y>1}=\frac{1}{2\pi i}\int_{c-iT}^{c+iT}\frac{y^s}{s}\,ds+
O^*\left(y^c\min\left(0.501,\frac{1}{\pi T|\log y|}\right)\right),\qquad y\in(0,1)\cup(1,\infty).\]
This formula follows by making explicit the calculation on \cite[pp.~105--106]{Da}. The constant $0.501$ follows by observing that the line $\Re s=c$ divides the circle $|s|=|c+iT|$ into two almost equal arcs, each of length less than $1.001\pi|c+iT|$. The constant $1/\pi$ arises as twice the size of the leading $1/(2\pi i)$. Applying the formula for $y=x/n$, multiplying by $\Lambda(n)$, and summing over $n\geq 1$, we get
\begin{equation}\label{psi1}
\psi(x)=
\frac{1}{2\pi i}\int_{c-iT}^{c+iT} \left(-\frac{\zeta'(s)}{\zeta(s)}\right) \frac{x^s}{s}\, ds
+O^*\left(\sum_{n=1}^\infty\left(\frac{x}{n}\right)^c\Lambda(n)
\min\left(0.501,\frac{1}{\pi T\left|\log\frac{x}{n}\right|}\right)\right).
\end{equation}

We shall abbreviate the integrand in \eqref{psi1} by $f(s)$, and estimate the error term by cutting the $n$-sum into four parts. Throughout, we keep in mind that $x/T<x/(x-2)$. As a preparation, we record the
elementary inequalities
\begin{equation}\label{psi2}
\left(\frac{x}{n}\right)^c\Lambda(n)\leq\left(\frac{x}{n}\right)^c\log n
=\frac{x}{n}\cdot\frac{e\log n}{e^{\log n/\log x}}\leq\frac{x}{n}\log x,
\end{equation}
\begin{equation}\label{psi3}
\left|\log\frac{x}{y}\right|=\int_{\min(x,y)}^{\max(x,y)}\frac{du}{u}
\geq\frac{|x-y|}{\max(x,y)},\qquad y>0.
\end{equation}
We also observe that the function
\[y\mapsto\frac{x-y}{y\log\frac{x}{y}},\qquad y>0,\]
is positive and decreasing (the function has a removable discontinuity at $y=x$). Indeed, writing $v:=\log\frac{x}{y}$, the claim is that $v\mapsto(e^v-1)/v$ is positive and increasing on $\RR$, which in turn follows from the fact that the exponential function is increasing and convex.

First we consider the $n$'s satisfying $1\leq|x-n|\leq\log x$. By \eqref{psi3} and the subsequent observation, in this range we have
\[0<\frac{x-n}{n\log\frac{x}{n}}\leq\frac{\log x}{(x-\log x)\log\frac{x}{x-\log x}}\leq\frac{x}{x-\log x}<1.001,\]
hence by \eqref{psi2} also
\begin{equation}\label{psi4}
\left(\frac{x}{n}\right)^c\frac{\Lambda(n)}{\left|\log\frac{x}{n}\right|}
\leq\frac{x\log x}{n\left|\log\frac{x}{n}\right|} < 1.001\frac{x\log x}{|x-n|}.
\end{equation}
So the corresponding $n$-subsum within \eqref{psi1} is at most
\begin{equation}\label{contr1}
1.001\frac{x\log x}{\pi T}\sum_{\substack{1\leq|x-n|\leq\log x\\\Lambda(n)\neq 0}}\frac{1}{|x-n|}
<0.638\log x\sum_{1\leq k\leq\log x}\frac{1}{k}<A(x),
\end{equation}
where
\[A(x):=0.638\log x\cdot(\log\log x+3/5).\]
Second, we consider the $n$'s satisfying $\log x<|x-n|\leq x/5$. In this range we have the following variant of \eqref{psi4} proved in the same way:
\[\left(\frac{x}{n}\right)^c\frac{\Lambda(n)}{\left|\log\frac{x}{n}\right|}
\leq\frac{x\log x}{n\left|\log\frac{x}{n}\right|}\leq\frac{1}{4\log\frac{5}{4}}\cdot\frac{x\log x}{|x-n|}.\]
So the corresponding $n$-subsum within \eqref{psi1} is at most
\begin{equation}\label{contr2}
\frac{1}{4\log\frac{5}{4}}\cdot\frac{x\log x}{\pi T}\sum_{\substack{\log x<|x-n|\leq x\\\Lambda(n)\neq 0}}\frac{1}{|x-n|},
\end{equation}
where we relaxed the summation condition for convenience. If $p(u)$ denotes the number of prime powers in $[x-u,x+u]$, then the last sum can be written as
\begin{equation}\label{contr3}
\sum_{\substack{\log x<|x-n|\leq x\\\Lambda(n)\neq 0}}\frac{1}{|x-n|}=\int_{\log x}^{x}\frac{dp(u)}{u}
=\left[\frac{p(u)}{u}\right]_{\log x}^{x}+\int_{\log x}^{x}\frac{p(u)}{u^2}\,du.
\end{equation}
Since $0\leq p(u)\leq 2u+1$, the first term on the right-hand side is less than $2.001$. By the Brun--Titchmarsh inequality in the form given by Montgomery and Vaughan~\cite[Th.~2]{MV2}, we also see that
\[p(u)<\frac{4u}{\log u}+\log_2(x)\left(1+\frac{2u}{\sqrt{2x}}\right),\qquad u\in[2,x].\]
Indeed, on the right-hand side, the first term upper bounds the number of primes in $[x-u,x+u]$, while the second term upper bounds the number of higher prime powers in $[x-u,x+u]$. In particular,
\begin{equation}\label{contr4}
\int_{\log x}^{x}\frac{p(u)}{u^2}\,du<
\int_{\log x}^{x}\left(\frac{4.001}{u\log u}+\frac{\log_2 x}{u^2}\right)\,du
<4.001\log\log x-3.818.
\end{equation}
We infer from \eqref{contr2}--\eqref{contr4} that the $n$'s satisfying $\log x<|x-n|\leq x/5$ contribute to \eqref{psi1} less than
\[B(x):=1.427\log x\cdot(\log\log x-2/5).\]
Now we turn to the $n$'s satisfying $|x-n|>x/5$. In this range we have $|\log (x/n)|>\log(6/5)$, hence by \Cref{zetalemma1} the corresponding $n$-subsum is at most
\[\sum_{|x-n|>x/5}\Lambda(n)\left(\frac{x}{n}\right)^c \frac{1}{\pi T\log(6/5)}
< \frac{ex}{\pi T\log(6/5)}\left|\frac{\zeta'(c)}{\zeta(c)}\right| < 4.746\log x.\]
Finally, there are two $n$'s satisfying $|x-n|<1$, and by \eqref{psi2} their contribution to \eqref{psi1} is less than
\[0.501\left(\frac{x}{x-1}+\frac{x}{x}\right)\log x < 1.003 \log x.\]
Collecting everything, we arrive at (with room to spare)
\begin{equation}\label{psi}
\psi(x) = \frac{1}{2\pi i}\int_{c-iT}^{c+iT} f +O^*(A(x)+B(x)+6\log x).
\end{equation}
We note here that $A(x)+B(x)$ is less than $2.1\log x\log\log x$.

On the other hand, the residue theorem combined with \eqref{eq:log-der-hline-high} and \Cref{zetalemma2}
shows that
\[\frac{1}{2\pi i}\left(\int_{c-iT}^{c+iT}f+\int_{c+iT}^{-\infty+iT}f+\int_{-\infty-iT}^{c-iT}f\right)
=x-\sum_{|\Im\rho|<T}\frac{x^\rho}{\rho}-\log(2\pi)-\frac{1}{2}\log(1-x^{-2}),\]
where each integral is over a directed line segment or half-line. We estimate the second integral with the help of \eqref{eq:log-der-hline-high} and \Cref{zetalemma2}:
\[\left|\int_{c+iT}^{-\infty+iT}f\right|
<\frac{\log^2 z + 20\log z}{z-2}\int_{-\infty}^{c}x^\sigma\,d\sigma<(\log x+20)\frac{ex}{x-2}.\]
The third integral obeys the same bound, hence we infer that
\[\frac{1}{2\pi i}\int_{c-iT}^{c+iT}f=x-\sum_{|\Im\rho|<T}\frac{x^\rho}{\rho}+O^*(\log x+20).\]
By \Cref{deltazerocount}, we can extend the $\rho$-sum to $|\Im\rho|<z$ at the cost of an error of $O^*(2\log x)$. Therefore, going back to \eqref{psi}, we conclude for $x\not\in\ZZ$ that
\[\psi(x)=x-\sum_{|\Im\rho|<z}\frac{x^\rho}{\rho}+O^*(A(x)+B(x)+9\log x+20).\]

Finally, if $x$ is an integer, then we make use of the following simple observation. For a fixed $z>10^{18}$, the $\rho$-sum on the right-hand side is continuous in $x\in(10^{18},z)$, while the left-hand side equals $\psi(x-)+\Lambda(x)$. Therefore, the previous formula is valid at $x$ with an extra error term of $O^*(\log x)$. We finish the proof of \Cref{dudeksharpened} by noting that
\[A(x)+B(x)+10\log x+20<5\log x\log\log x,\qquad x>10^{18}.\]

\section*{Acknowledgements} We thank the referee for valuable comments that helped us improve the exposition.

\section*{Conflict of interest and data availability statements}
On behalf of all authors, the corresponding author states that there is no conflict of interest.

Data sharing not applicable to this article as no datasets were generated or analysed during the current study.

\end{document}